\documentclass[draft]{amsart}

\usepackage{amsmath, amscd}
\usepackage{amsthm}
\usepackage{amssymb}
\usepackage{amscd, mathrsfs}
\usepackage[colorlinks]{hyperref}
\usepackage[all,dvips]{xy}

\newtheorem{theorem}{Theorem}[section]
\newtheorem{lemma}[theorem]{Lemma}
\newtheorem{prop}[theorem]{Proposition}
\newtheorem{cor}[theorem]{Corollary}

\newtheorem{fact}{Fact}

\theoremstyle{definition}
\newtheorem{definition}[theorem]{Definition}
\newtheorem{example}[theorem]{Example}

\newtheorem{remark}[theorem]{Remark}

\newtheorem{ex}[theorem]{Example}
 \numberwithin{equation}{theorem}
 
\newcommand{\coker}{\mathop{\mathrm{Coker}}\nolimits}

\newcommand{\N}{{\mathbb N}}
\newcommand{\Z}{{\mathbb Z}}

\newcommand{\rank}{\mathop{\mathrm{rank}}\nolimits}

\newcommand{\im}{\mathop{\mathrm{Im}}\nolimits}

\newcommand{\projdim}{\mathop{\mathrm{pd}}\nolimits}
\newcommand{\proj}{\mathop{\mathrm{Proj}}\nolimits}

\newcommand{\tor}{\mathop{\mathrm{Tor}}\nolimits}

\renewcommand{\Z}{\mathbb Z}

\newcommand{\longto}{\longrightarrow}
\renewcommand{\l}{\ell}
\newcommand{\8}{{\infty}}	
\newcommand{\T}{{\lceil}}
\newcommand{\4}{{\lfloor}}
\newcommand{\7}{{\rceil}}
\newcommand{\GG}{{\Gamma}}
\newcommand{\3}{{\rfloor}}

\renewcommand{\hom}{\mathop{\mathrm{Hom}}\nolimits}
\renewcommand{\O}{\mathcal{O}}

\newcommand{\E}{\epsilon(q)}
\renewcommand{\S}{\mathcal{S}}
\newcommand{\V}{\mathcal{V}}
\newcommand{\W}{\mathcal{W}}

\newcommand{\syz}{\mathop{\mathrm{Syz}}\nolimits}
\renewcommand{\i}[1]{\mathfrak{#1}}
\renewcommand{\ker}{\mathop{\mathrm{Ker}}\nolimits}

\renewcommand{\tilde}{\protect\widetilde}
\renewcommand{\bar}{\protect\overline}

\newcommand{\lra}{\longrightarrow}

\renewcommand{\l}{\ell}

\newcommand{\tsd}{\operatorname{t.s.d}}
\newcommand{\soc}{\operatorname{soc}}

\newcommand{\charac}{\mathrm{char}}
\newcommand{\length}{\lambda}

\begin{document}
\title[Asymptotic behavior of the socle of Frobenius powers]{Asymptotic behavior of the socle of Frobenius powers}
\author{JINJIA LI}
\address{Department of Mathematics\\ 328 Natural Sciences Building\\ University of Louisville\\ Louisville, KY 40292\\ USA}
\email{jinjia.li@louisville.edu}

\date{Received May 11, 2011; received in final form \today}


\begin{abstract}
Let $(R, \i m)$ be a local ring of prime characteristic $p$ and $q$ a varying power of $p$. We study the asymptotic behavior of the socle of $R/I^{[q]}$ where $I$ is an $\i m$ -primary ideal of $R$. In the graded case, we define the notion of diagonal $F$-threshold of $R$ as the limit of the top socle degree of  $R/\i m^{[q]}$ over $q$ when $q \to \infty$. Diagonal $F$-threshold exists as a positive number (rational number in the latter case) when: (1) $R$ is either a complete intersection or $R$ is $F$-pure on the punctured spectrum; (2) $R$ is a two dimensional normal domain. In the latter case, we also discuss its geometric interpretation and apply it to determine the strong semistability of the syzygy bundle of $(x^d, y^d,z^d)$ over the smooth projective curve in $\mathbb P^2$ defined by $x^n+y^n+z^n=0$. The rest of this paper concerns a different question about how the length of the socle of $R/I^{[q]}$ vary as $q$ varies. We give explicit calculations of the length of the socle of $R/\i m^{[q]}$ for a class of hypersurface rings which attain the minimal Hilbert-Kunz function. We finally show, under mild conditions, the growth of such length function and the growth of the second Betti numbers of $R/\i m^{[q]}$  differ by at most a constant, as $q \to \infty$.
\end{abstract}

\maketitle

\section{Introduction}
We first review some notation and 
definitions used throughout the paper.
In general, for a commutative ring $R$ of prime characteristic $p>0$, the Frobenius endomorphism $f\colon R \to R$ is defined by $f(r)=r^p$ for $r \in
R$; its self-compositions are given by $f^n(r) = r^{p^n}$. Restriction of scalars along each iteration $f^n$ endows $R$ with a new $R$-module structure, denoted by ${}^{f^n}\!\! R$.
For simplicity, we use $q$ to denote $p^n$. If $I$ is an ideal of $R$, the $q$-th Frobenius power of $I$ is the
ideal generated by the $q$-th powers of the generators of $I$, denoted by $I^{[q]}$. We use $F^n(-)$ to denote the functor from the category of $R$-modules to itself, given by base change along the Frobenius endomorphism $R \to {}^{f^n}\!\! R$. It is easy to see that $F^n(R/I) \cong R/I^{[q]}$. Also, the derived functors of $F^n(-)$ are  $\tor^R_i (-,{}^{f^n}\!\! R)$. For an $R$-module $M$, we use
 $\length(M)$ (resp. $\projdim M$) to denote the
length (resp. projective dimension) of $M$. When $R$ is local with maximal ideal $\i m$, the socle of $M$ is $(0: \i m)_M$, which is isomorphic to $\hom_R(R/\i m, M)$. For an $\i m$-primary ideal $I$, the Hilbert-Kunz function of $R$ with respect to $I$  is the length function $\length (R/I^{[q]})$ (as a function of $q$); the Hilbert-Kunz multiplicity of $R$ with respect to $I$ is the limit of $\length (R/I^{[q]})/q^{\dim R}$ as $q \to \8$. Such a limit always exists \cite{M1}.

In this paper, we investigate questions related to the following general question: how does the socle of $R/I^{[q]}$ vary as $q$ varies? While these questions are fairly well-understood when $I$ has finite projective dimension (see \cite{KV}), they remain quite mysterious when the projective dimension of $I$ is infinite and this is where our original motivation came from. In addition to that, our studies on these questions are also motivated by their relations with Hilbert-Kunz function and tight closure theory from many aspects. We refer to \cite{BC, B1, Y} for work along those lines.  

The organization of this paper is as follows. In Section~\ref{AB-tsd}, for a standard graded local algebra $(R, \i m)$ over a field of characteristic $p$, we define the notion of diagonal $F$-threshold $c^I(R)$ of $R$ (with respect to a homogeneous $\i m$-primary ideal $I$). It is the limit of the top socle degree of $R/I^{[q]}$ over $q$ as $q \to \8$. Such a definition agrees with the definition of $F$-threshold in the literature, which is defined under a more general set-up. The existence of the diagonal $F$-threshold of complete intersection rings follows from recent work of Kustin and Vraciu (see Proposition~\ref{TSD}). In general, it is not easy to calculate this invariant unless $I$ has finite projective dimension. However, in graded dimension two case (smooth projective curve case), we can use the geometric tools developed in \cite{B2} to study it and the rest of Section~\ref{AB-tsd} is devoted to this task. In graded dimension 2, the diagonal $F$-thresholds are all rational numbers. Specifically, we prove in Theorem~\ref{SS} that if the syzygy bundle of $I$ is strongly semistable and the degrees $d_1, \cdots, d_s$ of the generators of $I$ satisfy certain condition, then the diagonal $F$-threshold $c^I(R)$ is just $\frac{d_1+\cdots + d_s}{s-1}$, a rational number independent of the characteristic $p$. The case that the syzygy bundle of $I$ is not strongly semistable will be discussed in Theorem~\ref{NSS}. We show that in this case, under certain conditions, $c^I(R)$ is equal to the rational number $\nu_t$ appeared originally in \cite{B2},  whose definition (in general) relies on a result of Langer about the existence of the \textit{Strong Harder-Narashimhan filtrations}. As a result, we can use a numerical condition on the diagonal $F$-threshold to characterize the strong semistability of the syzygy bundle of $I$ when $I$ is generated by homogeneous elements of the same degree and $R$ is of the form $k[x,y,z]/(f)$ (Corollary~\ref{equiv}).

In Section~\ref{application}, we apply Corollary~\ref{equiv} to study the strong semistability of the syzygy bundle of $I=(x^d, y^d,z^d)$ over the smooth projective curve $\proj k[x,y,z]/(x^n+y^n+z^n)$ in prime characteristic $p$. Our work here relies heavily on a very recent preprint \cite{KRV} of Kustin, Rahmati and Vraciu, in which they completely determine how the property $\projdim I <\8$ depends on the parameters $p,n$ and $d$. We are able to transfer their results to determine, in quite many cases, how the strong semistability of the syzygy bundle of $I$ depends on the parameters $p,n$ and $d$. In particular, with some restrictions on $p,n$ or $d$, the strong semistability of the syzygy bundle of $I$ can be characterized by the condition $\projdim I^{[q]}=\8$ for all $q\gg0$.

In Section~\ref{sl} and Section~\ref{cl}, we study the the asymptotic behavior of length of socle of $R/I^{[q]}$. Such a length function had been preliminarily investigated by the author in his Ph.D thesis, for the purpose of answering a question of Dutta related to the nonnegativity conjecture of intersection multiplicity in the non-regular case, we refer to \cite{D} and \cite{L} for more details in that direction. Nevertheless, it is in general quite challenging to explicitly calculate this length function; even for the less complicated Hilbert-Kunz functions, the explicit calculations could be quite none trivial (for examples, see \cite{HM}).  The main result here is an explicit calculation  of this length function for a special class of hypersurface rings which attain the minimal Hilbert-Kunz function (see Definition~\ref{mHK}). What makes this calculation possible in such a case is the observation that the entire socle of $R/\i m^{[q]}$ lives in the top degree spot. We do not know how to calculate this length function, or merely determine a leading term,  when the socle contains elements of different degrees. We also point out a two-dimensional example in which the limit of $\length (\soc(R/\i m^{[q]})/ q^{\max\{0, \dim R-2\}})$ fail to exist as $q \to \8$.

In Section~\ref{betti}, we use Gorenstein duality and some spectral sequence arguments to prove Theorem~\ref{same}. In particular, it shows under mild conditions, the lengths of socle of $R/I^{[q]}$ and the second Betti numbers of $R/I^{[q]}$ differ only by a constant for $q$ sufficiently large.
\section{Asymptotic behavior of top socle degree of Frobenius powers}\label{AB-tsd}

Throughout this section, we assume ($R, \i m$) is a standard graded Noetherian local algebra over a field of positive characteristic $p$. Let $I$ be a homogeneous $\i m$-primary ideal  of $R$. Recall the $a$-invariant $a(R)$ of $R$ is the largest integer $m$ such that $(H_{\i m}^{\dim R}(R))_m\neq 0$, where $H_{\i m}^{\dim R}(R)$ is the top local cohomology module of $R$. When $R$ is complete intersection of the form $S/C$ where $S$ is the polynomial ring $k[x_1,\cdots, x_n]$ and $C$ is the ideal generated by homogeneous regular sequence $f_1, \cdots, f_t$, $a(R)=\sum_i \deg f_i -n$. For every standard graded Artinian $R$-module $M=\oplus M_n$, we use $\tsd( M)$ to denote the top socle degree of $M$, which is equal to $\max \{n| M_n \neq 0\}$.  Recently, Kustin and Vraciu (see \cite{KV}, Proposition 7.1) established the following lower bound for the top socle degree of $R/I^{[q]}$, in the case either $R$ is complete intersection or $R$ is Gorenstein and $F$-pure:

\begin{theorem}[Kustin-Vraciu] \label{KV}
Assume either $R$ is complete intersection or $R$ is Gorenstein and F-pure. If $\tsd (R/I)=s$, then for every $q$
\[\tsd (R/I^{[q]}) \geq  (s-a(R))q+a(R)\]
\end{theorem}

We remark here that since $R$ is standard graded, the powers $\i m^r$ are exactly $\oplus_{n\geq r} R_n$. Therefore the top socle degree of $R/I^{[q]}$ is nothing but the invariant
\[\nu_{\i m}^I(q) := \max\{r\in \N| {\i m}^r \not\subseteq I^{[q]} \}.\]
The limit (when exists) of  $\{\nu_{\i m}^I(q)/q\}$ as $q \to \8$ is a special case of an invariant called \textit{$F$-threshold}. More generally, for any ideals $\i a$ and $J$ of $R$ (not necessarily graded) with $\i a \subseteq \sqrt{J}$, one can define $\nu_{\i a}^J(q) = \max\{r\in \N| {\i a}^r \not\subseteq J^{[q]} \}$ and study the convergence of the sequence $\nu_{\i a}^J(q)/q$ as $q \to \8$. We refer to \cite{BMS, HMTW, HTW, MTW} for details on that direction. $F$-thresholds are known to exist for rings which is $F$-pure on the punctured spectrum. Here we focus on a special case of $F$-threshold which we will call \textit{diagonal $F$-threshold}. 

\begin{definition}
The \emph{diagonal $F$-threshold} of $R$ with respect to an $\i m$-primary ideal $I$, denoted $c^I(R)$, is defined as
\[ c^I(R) =\lim_{q \to \8} \dfrac{\tsd (R/I^{[q]})}{q}\]
whenever such a limit exists. We also use $c(R)$ to denote $c^{\i m}(R)$, and simply call it the diagonal $F$-threshold of $R$.
\end{definition}

\begin{prop} 
\label{TSD}
The diagonal $F$-threshold $c^I(R)$ exists when $R$ is complete intersection or is $F$-pure on the punctured spectrum. Moreover, 
if $R$ is complete intersection of the form $S/C$ where $S$ is the polynomial ring $k[x_1,\cdots, x_n]$ and $C$ is the ideal generated by a homogeneous regular sequence $f_1, \cdots, f_t$, assuming $J$ is the lift of $I$ in $S$, then  
\begin{equation}\label{lower}
c^I(R) \geq c^J(S)-\sum \deg f_i.
\end{equation}
In particular, $c(R) \geq -a(R)$, where $a(R)$ is the $a$-invariant.
\end{prop}

\begin{proof}
The argument for the existence of diagonal $F$-threshold (or more generally, the $F$-threshold) in the case of $F$-pure on the punctured spectrum can be found in \cite{HTW}. Assume $R$ is complete intersection.
By Theorem~\ref{KV}, the sequence 
\begin{equation}\label{seq}
\dfrac{\tsd (R/I^{[q]})-a(R)}{q}
\end{equation}
is increasing as $q$ increases. On the other hand, this sequence is bounded up by a very simple argument contained in \cite{B1}. We include that argument here for the sake of convenience. Choose $K$ large enough such that $\i m^K \subseteq I$ and let $L$ be the number of generators of $\i m^K$, then we have trivial inclusion $(\i m^K)^{Lq} \subseteq (\i m^K)^{[q]} \subseteq I^{[q]}$. This means $$\tsd (R/I^{[q]}) < (KL)q.$$ Therefore the sequence (\ref{seq}) converges as $q \to \8$, which implies 
the diagonal $F$-threshold $c^I(R)$ exists.

For the proof of the lower bound (\ref{lower}), we write $I$ as $I_1\cap I_2 \cap \cdots \cap I_b$ with each $I_i$ irreducible. Let $J_i$ be the lift of $I_i$ in $S$. It is easy to check that $c^I(R) \geq c^{I_i}(R)$ for each $i$. Since the Frobenius endomorphism on $S$ is flat,  has $c^J(S) = \max \{ c^{J_i}(S) |i\}$. Thus we reduce (\ref{lower}) to the case that $I$ is irreducible. In such a case, we have (see \cite{KV}, page 206)
\begin{equation}\label{KVrelation}
\tsd(\dfrac{S}{J^{[q]}+C})=\tsd(\dfrac{S}{J^{[q]}})-M_q
\end{equation}
where $M_q$ is the least degree among homogeneous nonzero elements of $(J^{[q]}:C)/J^{[q]}$,  which is less than or equal to $(q-1)\sum \deg f_i$. Then the lower bound (\ref{lower}) follows from dividing both sides of (\ref{KVrelation}) by $q$ and taking the limit. In particular, when $I$ is the maximal ideal $\i m$ of $R$, $J$ is the maximal ideal of $S$. In this case, since $c^J(S)=\dim S =n$, the right hand side of (\ref{lower}) is $n-\sum \deg f_i=-a(R)$.
\end{proof}

If $R$ is Gorenstein and $F$-pure, Theorem~\ref{KV} also provides a lower bound for diagonal $F$-thresholds, namely,
$$
c^I(R) \geq \tsd (R/I)-a(R).
$$
In particular, when $I= \i m$, one has $c(R) \geq -a(R)$.
Such a bound is achievable. For example, let $R=k[x,y,z,w]/(xy-zw)$, then $c(R)=2$ and $a(R)=-2$.

\begin{remark}
For a Cohen-Macaulay normal domain $R$, Brenner gave an upper bound of $\tsd (R/I^{[q]})$ which is better than the trivial upper bound used in the proof of Proposition~\ref{TSD}. We refer to \cite{B1} for details. When $I=\i m$, there is a lower bound $m(q)$ for $\tsd (R/\i m^{[q]})$ given by Buchweitz and Chen, see Theorem~\ref{BC} below. This is a uniform lower bound for all graded hypersurfaces of the form  $k[x_0,x_1, \cdots, x_n]/(f)$, where $f$ runs through all homogeneous polynomials of degree $d$. In particular, this implies $c(R)\geq \frac{n+1}{2}$.
\end{remark}


What information is encoded by the diagonal $F$-threshold $c^I(R)$ (or more generally, by $F$-thresholds)? Questions of this kind have been studied by many authors from many different angles (for example, see \cite{HMTW, HTW}). In the remaining part of this section, we investigate this question in the case of smooth projective curves. We first briefly recall some basic definitions.

{\bf{Harder-Narasimhan filtrations}}

Let $Y$ be a smooth projective curve over an algebraically closed
field. For any vector bundle $\V$ on $Y$ of rank $r$, the degree of $\V$ is defined as the degree of the line bundle $\wedge^r \V$. The 
slope of $\V$, denoted $\mu(\V)$, is defined as the fraction $\deg(\V)/r$. Slope is additive on tensor products of bundles: $\mu(\V \otimes \W)=
\mu(\V)+\mu(\W)$. If $f: Y'\lra Y$ is a finite map of degree $q$,
then $\deg (f^*(\V))=q\deg(\V)$ and so $\mu(f^*(\V))=q\mu(\V).$

A bundle $\V$ is called \emph{semistable} if for every subbundle
$\W\subseteq \V$ one has $\mu(\W) \leq \mu(\V)$. Clearly,
bundles of rank 1 are always semistable, and duals and twists of
semistable bundles are semistable.

Any bundle $\V$ has a filtration by subbundles
\[
0=\V_0 \subset \V_1 \subset \cdots \subset \V_t=\V
\]
such that $\V_k/\V_{k-1}$ is semistable and $\mu(\V_k/\V_{k-1}) >
\mu(\V_{k+1}/\V_k)$ for each $k$. This filtration is unique, and
it is called the \emph{Harder-Narasimhan (or HN) filtration} of $\V$.

In positive characteristic, we use $F$ to denote the absolute Frobenius morphism $F: Y \lra Y$. Pulling back a vector bundle under $F$ 
does not necessarily preserve semistability. Therefore, the
pullback under $F^n$ of an HN filtration of
$\V$ does not always give an HN filtration of
$(F^*)^n(\V)$. However, by the work of Langer \cite{La}, there always exists a so called strong NH filtration,
i.e., for some $n_0$, the HN filtration of $(F^*)^{n_0}(\V)$ has the property that all its
Frobenius pullbacks are the HN filtrations of $(F^*)^n(\V)$, for all
$n>n_0$. 

Suppose $R$ is a standard-graded two-dimensional normal domain and $I=(f_1,\cdots, f_s)$ where $f_i$ is homogeneous of degree $d_i$ for $1\leq i \leq s$.
Let $Y=\proj R$. Consider the syzygy
bundle $\S=\syz(f_1, \dots, f_s)$ on $Y$ given by the exact sequence
\begin{equation}
\label{def-syzygy}
0\lra \S \lra \bigoplus_{i=1}^s \O(-d_i) 
\overset{f_1, \ldots, f_s} {\longto} \O \lra 0
\end{equation} 
and the pullback of this exact sequence along 
$F^n$ (with a subsequent
twist by $m\in \mathbb Z$)
\begin{equation}
\label{exact-seq}
0\lra \S^q(m) \lra \bigoplus_{i=1}^s \O(m-qd_i) 
\overset{f_1^q, \dots, f_s^q} {\lra} \O(m) \lra 0
\end{equation} 
where $\S^q$ denotes the pullback $(F^*)^n(\S) = \syz (f_1^q, \dots, f_s^q)$.  Applying the sheaf cohomology to (\ref{exact-seq}), we have a long exact sequence
\begin{equation}\label{les}
0\to H^0(Y, \S^q(m)) \to \bigoplus_{i=1}^s H^0(Y, \O(m-qd_i)) \stackrel{f_1^q
, \ldots, f_s^q} {\longto} H^0(Y, \O(m)) \to
  \end{equation}
 \[ \to H^1(Y, \S^q(m)) \to \bigoplus_{i=1}^s H^1(Y,  \O(m-qd_i)) \to \cdots.\]
As $R$ is normal, the cokernel of the third map (from left)
in the long exact sequence (\ref{les})
is the $m$-th graded piece of $R/I^{[q]}$.


Now we are ready to move back to the question we asked earlier in this section. Set $\deg Y= \deg \O_Y(1)$ and let $\omega_Y$ denote the dualizing sheaf on $Y$.
We first treat the case where the syzygy bundle is strongly semistable. 

\begin{theorem} \label{SS}
Suppose the syzygy bundle $\S$ is strongly semistable. Assume also the degrees $d_i$ satisfy the condition $\frac{d_1+\cdots+d_s}{s-1} > \max_i\{d_i\}$, then 
\[c^I(R)= (d_1+\cdots + d_s)/(s-1).\]
\end{theorem}

\begin{proof}
The rightmost term in (\ref{les}) is zero for $m > \max_i\{qd_i\}$. 
This is because, by Serre duality,  $h^1(\O(m-qd_i))=h^0(\O(-m+qd_i)\otimes \omega_Y)$, which equals zero since the degree of 
$\O(-m+qd_i)\otimes  \omega_Y$ is negative.

From \cite{B2}, we know for $m > \T \frac{d_1+ \cdots d_s}{s-1}q \7 +\frac{\deg \omega_Y}{\deg Y}$, $ H^1(Y, \S^q(m))$ vanishes. It follows that 
\[\tsd (R/I^{[q]}) \leq \T \frac{d_1+ \cdots d_s}{s-1}q \7 +\dfrac{\deg \omega_Y}{\deg Y}.\]
Also from \cite{B2}, for $m \leq \T \frac{d_1+ \cdots d_s}{s-1}q \7 -1$,  $ H^0(Y, \S^q(m))$ vanishes. An easy calculation asserts that, for $m=\T \frac{d_1+ \cdots d_s}{s-1}q \7 -1$ and $q\gg0$,
\[\sum_{i=1}^s h^1(\O(m-qd_i) \neq h^0(\O(m)).\]
Thus 
\[\tsd (R/I^{[q]}) \geq \T \frac{d_1+ \cdots d_s}{s-1}q \7-1, \text{ for } q \gg0,\]
and the theorem follows.
\end{proof}

We next discuss the case where the syzygy bundle is not strongly semistable. In such a case, using strong NH filtrations, Brenner defined rational numbers $\nu_1, \cdots, \nu_t$ for the syzygy bundle $\S$
\[\nu_i=-\dfrac{\mu(F^{*n}(\S_i))/\mu(F^{*n}(\S_{i-1}))}{q\deg \O(1)},\]
where $0=\S_0 \subset \S_1 \subset \cdots \subset \S_t=\S$ is a HN filtration of $F^{*{n_0}}(\S)$ which is strong and $q=p^{n_0+n}$.
These $\nu_i$'s satisfy $\min\{d_i\} \leq \nu_1 < \cdots <\nu_t \leq \max \{d_i+d_j|i \neq j\}$. Moreover,
 Brenner showed for $q\gg0$, if $m>q\nu_t+\frac{\deg \omega_Y}{\deg Y}$, then $ H^1(Y, \S^q(m))=0$ (see \cite{B2} for more details).
 
Let $g$ denotes the genus of $Y$. With the above set-up, we have 
\begin{theorem}\label{NSS}
\begin{itemize}
 \item[(1)] $\tsd (R/I^{[q]}) \leq\nu_tq+\frac{\deg \omega_Y}{\deg Y}, \text{ for } q\gg0.$
 \item[(2)] If we further assume 
$\nu_t> \max_i\{d_i\}$,
then for $q\gg0$, 
\[ \tsd (R/I^{[q]}) \geq \begin{cases}
\T\nu_t q\7-1, & \text{ if } g \geq 1 \\
\T\nu_t q\7-2, & \text{ if } g=0
\end{cases}\]
In particular, $c^I(R)$ exists and equals $\nu_t$.
\end{itemize}
\end{theorem}

\begin{proof}
(1) Let $l= \tsd (R/I^{[q]})$. Thus the cokernel of 
\[\bigoplus_{i=1}^sH^0(Y, \O(l-qd_i)) \stackrel{f_1^q
, \ldots, f_s^q} {\longto} H^0(Y, \O(l))\]
must be nonzero, which implies $ H^1(Y, \S^q(m)) \neq 0$. This shows $l \leq q\nu_t+\frac{\deg \omega_Y}{\deg Y}$ when $q\gg0$. 

(2) Here we only treat the case $g\geq1$, the computations for the other case are almost identical. So we assume $g\geq 1$. 
To get the lower bound $\T\nu_t q\7-1$ in this case, let  $\l=\T\nu_t q\7-1$, it suffices to show the cokernel of 
\[\bigoplus_{i=1}^sH^0(Y, \O(\l-qd_i)) \stackrel{f_1^q
, \ldots, f_s^q} {\longto} H^0(Y, \O(\l))\]
is nonzero. To this end, we apply the following result from \cite{B2}, page 102:

For $q\gg0$ and $q\nu_{t-1}+\frac{\deg \omega_Y}{\deg Y} <m < q\nu_t,$
\begin{align}\label{e25}
h^0(Y, \S^q(m))=q(-r_1\nu_1- &\cdots - r_{t-1} \nu_{t-1})\deg Y+\\
&+m(r_1+\cdots + r_{t-1})\deg Y+\rank(\S_{t-1})(1-g). \notag
\end{align}
Here, $r_i$ is defined to be the rank of $\S_i/\S_{i-1}$ for $i=1, \cdots, t$. These numbers $r_1, \cdots, r_t$ satisfy
\begin{equation}\label{rank}
r_1+\cdots +r_t=\rank \S=s-1
\end{equation}
and
\begin{equation}\label{rankslope}
r_1\nu_1+\cdots +r_t\nu_t=\sum_{i=1}^s d_i.
\end{equation}
Let $\E=\T\nu_t q\7-\nu_tq$, then $\l=\nu_tq-1+\E$. Therefore, by (\ref{e25}), we have
\begin{align}\label{e25l}
h^0(Y, \S^q(\l))=q(-&r_1\nu_1- \cdots - r_{t-1} \nu_{t-1})\deg Y+\\
&+(\nu_tq-1+\E)(r_1+\cdots + r_{t-1})\deg Y+\rank(\S_{t-1})(1-g) \notag\\
=q(\nu_t&(s-1-r_t)-r_1\nu_1- \cdots - r_{t-1} \nu_{t-1})\deg Y+\notag\\
&+(-1+\E)(s-1-r_t)\deg Y+\rank(\S_{t-1})(1-g). \notag
\end{align}
On the other hand, since $\nu_t>\max\{d_i\}$, we have $-\l+qd_i<0$ for $q\gg0$, whence 
\[H^1(Y, \O(\l-qd_i))=H^0(Y, \O(-\l+qd_i))\otimes \omega_Y)=0, \text{ for }q\gg0.\]
So from Riemann-Roch Theorem, we get
\begin{equation}\label{e26}
h^0(\O(\l))-\sum_{i=1}^s h^0( \O(\l-qd_i))=
\l\deg Y +(1-g)-\sum_{i=1}^s (\l-qd_i)\deg Y-s(1-g)
\end{equation}
which can be simplified to
\begin{equation}\label{h0}
(s-1)(g-1)-\bigg((s-1)\l-\sum_{i=1}^s d_iq\bigg)\deg Y
=q\bigg(\sum_{i=1}^s d_i-(s-1)\nu_t\bigg)\deg Y+(s-1)\bigg((1-\E)\deg Y +g-1\bigg)
\end{equation}

Thus, by (\ref{les}), the length of the cokernel of 
\[\bigoplus_{i=1}^sH^0(Y, \O(\l-qd_i)) \stackrel{f_1^q
, \ldots, f_s^q} {\longto} H^0(Y, \O(\l))\]
equals
\begin{equation}\label{simple}
h^0(Y, \S^q(\l))+h^0(\O(\l))-\sum_{i=1}^s h^0( \O(\l-qd_i))=c_1q+c_0
\end{equation}
where, by adding the right hand sides of (\ref{e25l}) and (\ref{h0}) and using (\ref{rankslope}), 
\begin{align*}
c_1&=\deg Y\bigg(\nu_t(s-1-r_t)-r_1\nu_1- \cdots - r_{t-1} \nu_{t-1}+\sum_{i=1}^s d_i-(s-1)\nu_t\bigg)=0\\
\end{align*}
and 
\begin{align*}
c_0=&(-1+\E)(s-1-r_t)\deg Y+\rank(\S_{t-1})(1-g)+(s-1)((1-\E)\deg Y +g-1)\\
=&(-1+\E)(-r_t)\deg Y+\rank(\S_{t-1})(1-g)+(s-1)(g-1)\\
=&(1-\E)r_t\deg Y+(s-1-\rank(\S_{t-1}))(g-1).
\end{align*}
Since $\rank(\S_{t-1})+r_t=\rank\S=s-1$,
$$c_0=r_t((1-\E)\deg Y+g-1)>0.$$
The last inequality here is due to our assumption $g \geq 1$.
Therefore, the left hand side of (\ref{simple}) is positive for $q\gg0$.

\end{proof}

\begin{cor} \label{equiv}
Assume $R$ is of the form $k[x,y,z]/(f)$ such that the genus of $\proj R$ is at least one. $I=(f_1,f_2,f_3)$ and $\deg f_i=d$. Then $c^I(R) \geq \frac{3d}{2}$, and $c^I(R) = \frac{3d}{2}$ if and only if the syzygy bundle of $I$ is strongly semistable.  
\end{cor}

\begin{proof}
If the syzygy bundle of $I$ is strongly semistable, then $c^I(R) = \frac{3d}{2}$ follows from Theorem~\ref{SS}. 

Assume the syzygy bundle of $I$ is not strongly semistable. Here we have $t=2$, $\nu_1<\nu_2$ and $\nu_1+\nu_2=3d$. Therefore $c^I(R) =\nu_2 > \frac{3d}{2}$.
\end{proof}

\begin{remark}
Assume $R$ is of the form $k[x,y,z]/(f)$ where the degree of $f$ is $h$. By Corollary 4.6 in \cite{B2}, we have 
\[\frac{3}{2}\leq c(R) \leq 2\]
and $c(R)=3/2$ if and only if the syzygy bundle of $(x,y,z)$ is strongly semistable.
Moreover, the Hilbert-Kunz multiplicity and the diagonal $F$-threshold of $R$ have the relation
\[e_{\text{HK}}(R)=h(c(R)^2-3c(R)+3).\]
\end{remark}

\begin{example}\label{favorite}
Consider $R_p=\Z/p\Z[x,y,z]/(x^4+y^4+z^4)$. The rank two syzygy bundle $\S$ of $(x,y,z)$ is strongly semistable if $p \equiv \pm 1 \mod 8$ and therefore, $c(R_p)=\frac{3}{2}$ in this case. On the other hand, $\S$ is not strongly semistable if $p \equiv \pm 3 \mod 8$ (See \cite[Example 4.1.8]{K}, \cite{B3,M2} or Example~\ref{favorite1}). So
$$
c(R_p)=\nu_t=\nu_2=
\frac{3}{2}+\frac{1}{2p}.$$

Our Macaulay 2 experiments give us the following formulae of the top socle degree functions for $p=3,5,7$:
$$
\tsd(R_p/\i m^{[q]})=
\begin{cases}
\frac{5q}{3}+1, \text{ if } p = 3\\
\frac{8q}{5}+1, \text{ if } p =5\\
\lfloor \frac{3q}{2}\rfloor+1,  \text{ if } p =7.
\end{cases}
$$
\end{example}

One might expect the following precise formula of $\tsd (R/\i m^{[q]})$, depending only on $c(R)$ and $a(R)$,  
\[\tsd (R/\i m^{[q]})=\lfloor c(R)q \rfloor+a(R).\]

However, the following example, suggested by Brenner to the author, indicates this is wrong.
\begin{example}\label{weird}
Let $R=\Z/2\Z[x,y,z]/(x^4+y^4+z^4+x^3y+y^3z+z^3x)$. Then
$$\tsd(R/\i m^{[q]})=\frac{3}{2}q.$$
\end{example}
Example~\ref{weird} also shows the inequality in Theorem~\ref{KV} could be strict for all $q$.

\section{Some applications}\label{application}
Throughout this section, let $R$ be the diagonal hypersurface $k[x,y,z]/(x^n+y^n+z^n)$ where $\charac k=p$ and $I$ be the ideal $(x^d,y^d,z^d)$ of $R$. In a recent paper \cite{KRV}, Kustin, Rahamati and Vraciu completely determined how the property $\projdim I< \8$ depends on the parameters $p,n$ and $d$. For every prime number $p$, they introduced the sets $S_p$ and $T_p$, which form a partition for the set of all nonnegative integers. One of their main results in that paper is that  $I$ has finite projective dimension if and only if $n \mid d$ or $\4\frac{d}{n}\3 \in T_p$ (see \cite{KRV} Theorem 6.2). In addition to that, they also explicitly described the minimal free resolutions for such ideals. Our purpose here is to use the results in \cite{KRV}, together with the characterization of the strong semistability of the syzygy bundle of $I$ obtained in Corollary~\ref{equiv}, to study how the strong semistability of the syzygy bundle of $I$ depends on parameters $p, n$ and $d$. We would like to point out that the strong semistability of this particular syzygy bundle has also been studied in great detail in \cite{K}, Chapter 4. In the rest of this section, we adopt all the notation of \cite{KRV} without explanation and refer to \cite{KRV} for details. Since we will apply Corollary~\ref{equiv}, we also assume $p \nmid n$ and $n \geq 3$ throughout this section. With this assumption, $\proj R$ is a smooth curve. We first prove a sufficient condition for the strong semistability of syzygy bundle.
\begin{theorem}\label{fpd}
If there are infinitely many $q$ such that $\projdim I^{[q]} = \8$, then the syzygy bundle of $I$ is strongly semistable.
\end{theorem} 

\noindent We remark here that this sufficient condition is equivalent to: (a) $n \nmid d$ and; (b)  there are infinitely many $q$ such that $\4\frac{qd}{n} \3\in S_p $.

\begin{proof}[Proof of Theorem~\ref{fpd}]
We know by Theorem 3.5 of \cite{KRV}, when $\projdim I^{[q]} = \8$, the leading term of $\tsd (R/I^{[q]})$ is $(\frac{3d}{2})q$. Therefore, if there are infinitely many $q$ such that $\projdim I^{[q]} = \8$, then $c^I(R)=\frac{3d}{2}$. It then follows from Corollary~\ref{equiv} that the syzygy bundle of $I$ is strongly semistable.
\end{proof}

We next discuss the converse of Theorem~\ref{fpd}. Assume $\projdim I^{[q_0]} < \8$ for some $q_0$ (hence $\projdim I^{[q]} < \8$ for all $q \geq q_0$, see \cite{PS}). We know from Corollary 1.7 of \cite{KV}, $c^{I^{[q_0]}}(R)$ equals the largest back twist in the minimal homogeneous resolution of $R/I^{[q_0]}$  by free $R$-modules, i.e., the largest $b_i$ in the following resolution of $R/I^{[q_0]}$:
\[0 \to R(-b_1)\oplus R(-b_2) \to \bigoplus_{i=1}^3 R(-q_0d) \to R \to 0.\]
Since $b_1+b_2-2q_0d=q_0d$, we have:
$$b_1 \neq b_2\Leftrightarrow \max\{b_1, b_2\} >\frac{3q_0d}{2}$$
i.e.
\[c^{I^{[q_0]}}(R)>\frac{3q_0d}{2},\] 
which is equivalent to (observe that $c^I( R)=\frac{1}{q_0}c^{I^{[q_0]}} (R)$)
$$c^I( R)>\frac{3d}{2}.$$
Thus, by Corollary~\ref{equiv}, we have the following theorem.
\begin{theorem}\label{twist}
Assume $\projdim I^{[q_0]} < \8$ for some $q_0$. Then the syzygy bundle of $I$ is strongly semistable if and only if $b_1=b_2$.
\end{theorem}

\noindent The following example shows the converse of Theorem~\ref{fpd} does not hold in general. 

\begin{example}
Let $k=\Z/5\Z$, $n=3$ and $d=2$. Then $\projdim I^{[p]} < \8$ (hence $\projdim I^{[q]} < \8$ for all $q>p$). On the other hand, one can check the largest back twist in the minimal homogeneous resolution of $R/I^{[p]}$ is $15$ (i.e, $b_1=b_2=15$). So the syzygy bundle of $I$ is strongly semistable.
\end{example}

\noindent Nonetheless, we could still expect the converse of Theorem~\ref{fpd} to hold when the parameters $p, n, d$ satisfy special conditions. We first treat the case when $n \nmid d$:
\begin{theorem}\label{conv1}
Assume $n \nmid d$. Then the converse of Theorem~\ref{fpd} holds for all the following cases:
\begin{itemize}
\item[Case 1:]$p=3$;
\item[Case 2:]$p \equiv 1 \mod 3$; 
\item[Case 3:]$p$ and $d$ are both odd;
\item[Case 4:]$3 \nmid n$.
\end{itemize}
In particular, if $p=2$, $3 \nmid n$ and $n \nmid d$, then the syzygy bundle of $I$ is never strongly semistable.
\end{theorem}

\begin{proof} When $\projdim I^{[q_0]}<\8$, the minimal free resolution of  $I^{[q_0]}$ has been given explicitly in \cite{KRV}, section 5. It is then a bookkeeping job (but a little tedious) to check case by case that, in any of the above cases, $b_1 \neq b_2$ (we refer to particularly 5.3, 5.4, 5.5, 5.12, 5.13, 5.14, 5.15 of \cite{KRV}). Therefore by Theorem~\ref{twist}, the syzygy bundle of $I$ is not strongly semistable. The $p=2$ case of this theorem follows from the fact that $T_2$ is the set of all nonnegative integers, hence $ \projdim I^{[q]}$ is never finite in such a case.
\end{proof}

\begin{example}\label{favorite1} We apply this theorem to recover the known result we mentioned earlier in Example~\ref{favorite}: Suppose $p$ is odd. The syzygy bundle of $(x,y,z)$ over $\Z/p\Z[x,y,z]/(x^4+y^4+z^4)$ is strongly semistable if and only if $p \equiv \pm 1 \mod 8$. Here we have $n=4$ which is not divisible by $3$, so Theorem~\ref{conv1} is applicable. First consider the case $p=8c+1$. We claim that $\4\frac{p^e}{n} \3\in S_p$ for every positive integer $e$. To see this, notice that $\4\frac{p^e}{4} \3=\frac{p^e-1}{4}$. We can express this number in the following base-$p$ expansion (see \cite{KRV} Notation 1.5) 
$$\frac{p^e-1}{4}=2cp^{e-1}+2cp^{e-2}+ \cdots +2cp+2c, $$
which involves only even digits. So the claim follows from \cite{KRV} Remark 1.6. Therefore, the syzygy bundle is strongly semistable in this case. For the case $p=8c+3$, we claim that  $\4\frac{p^2}{4} \3\in T_p$. This is because $\4\frac{p^2}{4} \3=16c^2+12c+2=(2c+1)p-(2c+1)$, which is a base-$p$ expansion involves at least one odd digit. Thus the syzygy bundle is not strongly semistable. We leave the remaining cases for the interested readers to verify.
\end{example}

One can also apply the above base-$p$ expansion method to reinvestigate the strong semistability of the syzygy bundle $\S$ of $(x^2,y^2,z^2)$ on the Fermat quintic $x^5+y^5+z^5=0$ (here $n=5$ and $d=2$ so Theorem~\ref{conv1} is applicable), which has been studied in \cite{B5}, Section 2 and in \cite{K}, Example 4.1.9. In particular, one can recover Corollary 2.1 of \cite{B5}, which says $\S$ is not strongly semistable when $p \equiv \pm 2 \mod 5$, by looking at the base-$p$ expansions of $\4 \frac{p^2}{5} \3$ of such primes. Not only that, one also obtains that $\S$ is strongly semistable when $p \equiv \pm 1 \mod 5$ via the same method. Again, we leave the detail for the interested readers.

For the case $n \mid d$, we know there exists a $q_0\gg0$ such that  $\projdim I^{[q_0]} < \8$. In such a case, the strong semistability is determined by the syzygy gap. We refer to \cite{M3} or \cite{BK} for the definition of syzygy gap.
\begin{theorem}\label{conv2}
Assume $n \mid d$. Let $a=\frac{d}{n}$. The syzygy bundle of $I$ is strongly semistable if and only if the syzygy gap of $x^a,y^a,(x+y)^a$ in $k[x,y]$ is zero. In particular, if $a$ is odd, then the syzygy bundle of $I$ is never strongly semistable.
\end{theorem}
\begin{proof}
Let $\delta$ be the syzygy gap of $x^a,y^a,(x+y)^a$ in $k[x,y]$. 
From \cite{KRV} Observation 5.3, we know $b_1 = b_2$ if and only if $\delta=0$. So our conclusion follows from Theorem~\ref{twist}. There is a formula $\delta^2=4\l(k[x,y]/(x^a,y^a,(x+y)^a))-3a^2$ (\cite{M3}, Lemma 1). Hence if $a$ is odd, $\delta \neq 0$.
\end{proof}

In the case $n=3$ and $p \equiv 1 \mod 3$, we have the following characterization of strong semistability.
\begin{theorem}
Assume $n=3$, $p \equiv 1 \mod 3$ and $3 \nmid d$. Then the syzygy bundle of $I$ is strongly semistable if and only of $\projdim I=\8$.
\end{theorem} 
\begin{proof}
This follows from Proposition 8.5 of \cite{KRV}, Theorem~\ref{fpd} and Theorem~\ref{conv1} immediately. 
\end{proof}
\noindent We finally use Theorem~\ref{conv1} to recover a result of Brenner (Proposition 1 in \cite{B4}, see also Lemma 4.2.8 in \cite{K} for an extended version). 
\begin{theorem}Fix $d>0$ and a prime number $p$. For every positive integer $n_0$, there exists $n>n_0$ such that the syzygy bundle of $I=(x^d,y^d,z^d)$ is not strongly semistable on the smooth projective curve defined by $x^n+y^n+z^n=0$.
\end{theorem}
\begin{proof}
By Remark 1.6 in \cite{KRV}, the set $T_p$ is not empty. Assume $c \in T_p$. For every $n_0$, we choose $q$ large enough such that $\frac{qd}{c}>\max\{n_0+4, 4c+4, d+4\}$. Let $a$ be an integer such that $a\leq \frac{qd}{c} <a+1$. If $3 \mid a-1$, let $n$ be one of $a-2,a-3$ which is not divisible by $p$. If $3 \nmid a-1$, let $n$ be one of $a-2, a-1, a$ which is neither divisible by $3$ nor divisible by $p$. Then such an $n$ satisfies conditions $3 \nmid n$, $n \nmid qd$ and $\4\frac{qd}{n}\3 =c \in T_p$. Thus the conclusion follows from Theorem~\ref{conv1}.
\end{proof}

\section{Asymptotic behavior of length of socle of Frobenius powers}\label{sl}
Our main question in this section is:

How does $\length (\soc (R/I^{[q]}))$ vary as $q$ vary?

In the case that $I$ has finite projective dimension, it is well-known that $\length (\soc (R/I^{[q]}))$ equals the constant $\length (\soc (R/I))$.   
In the general situation, we recall a result of Yackel \cite{Y},  which asserts such a length function cannot grow faster than $O(q^{\max\{\dim R-2,0\}})$.
\begin{theorem}(Yackel)\label{Yackel}
There exists a constant $c_R$, s.t.
\[\lambda (\soc(R/I^{[q]})) \leq c_Rq^{\max\{0, n-2\}},\]
where $n=\dim R$.
\end{theorem}

Unfortunately, other than the above result of Yackel, very little is known regarding the asymptotic behavior of this length function when $I$ has infinite projective dimension, even in the hypersurface situation. Therefore, any explicit computation on such a length function in special cases would be valuable, and even could potentially initiate some new study. The main purpose of this section is to explicitly calculate such length functions for $I=\i m$ over a special class of hypersurface rings, which will be specified below. 

In \cite{BC}, Buchweitz and Chen investigated a lower bound $m(q)$ of the top socle degree of $R/\i m^{[q]}$ among all dimension $n$ hypersurfaces $(R, \i m)$, and its relation with minimal Hilbert-Kunz function. Among other things, they showed the following:

\begin{theorem}[Buchweitz-Chen]\label{BC} Fix $n, d>0$. Let $q$ be a power of $p$. Let 
\[m(q)=\bigg\lfloor \dfrac{(n+1)(q-1)+(d-1)}{2}\bigg \rfloor\]  
and \[L(q)= \text{ the coefficient of } t^{m(q)} \text{ in } (1-t^d)(1-t^q)^{n+1}(1-t)^{-n-2}.\]
Assume $f$ is a homogeneous polynomial of degree $d$ in $S=k[x_0,x_1, \cdots, x_n]$. Let $R$ be the hypersurface ring $S/fS$ and $\i m$ the ideal of $R$ generated by the images of $x_0, x_1, \cdots, x_n$ in $R$. Then the top socle degree $\tsd (R/\i m^{[q]})$ is at least $m(q)$ and the Hilbert-Kunz function of $R$ with respect to $\i m$ is at least $L(q)$. Moreover, the following are equivalent:
\begin{itemize}
\item[(1)] $\tsd (R/\i m^{[q]})=m(q)$;
\item[(2)] The Hilbert-Kunz function of $R$ with respect to $\i m$ equals $L(q)$.
\end{itemize} 

\end{theorem} 

\begin{definition}\label{mHK}
We say a hypersurface ring \emph{attains the minimal Hilbert-Kunz fuction} if it satisfies any of the equivalent conditions in the above theorem of Buchweitz and Chen.
\end{definition}

\begin{ex}[Buchweitz-Chen] The hypersurfaces $k[x,y,z,w]/(xy-zw)$ and the Cayley's cubic surface $k[x,y,z,w]/(xyz+xyw+xzw+yzw)$ both attain the minimal Hilbert-Kunz function, see \cite{BC} for details.
\end{ex}

The following theorem is the main result of this section. 
\begin{theorem}
\label{main}
Adopt all of the notation of Theorem~\ref{BC}. Assume $R=S/fS$ attains the minimal Hilbert-Kunz function. Then the entire socle of $R/\i m^{[q]}$ must live in the top degree spot. In other words, no socle element of $R/\i m^{[q]}$ is in degree $< m(q)$.
\end{theorem}

\begin{proof}
We use $\bold x^{[q]}$ to denote the ideal $(x_0, \cdots, x_n)^{[q]}$ of $S$. Let
\[\boldsymbol{\Theta}=\dfrac{S}{\bold x^{[q]}}=\bigoplus_{i\geq 0} \boldsymbol{\Theta}_i\]
and
\[\theta=\dfrac{S}{fS+\bold x^{[q]}}=\bigoplus_{i\geq 0} \theta_i.\]
Since we assume $R$ attains the minimal Hilbert-Kunz function, by the argument contained in the proof of Theorem~\ref{BC} in \cite{BC}, we have short exact sequences
\begin{equation}\label{ses}
0 \to  \boldsymbol{\Theta}_{i-d} \overset{f}{\to} \boldsymbol{\Theta}_i \to \theta_i \to 0
\end{equation}
for every $i \leq m(q)$.

Assume $i \leq m(q)$. Consider the following commutative diagram with exact rows
\begin{equation}
\begin{CD}
0 @>>>\boldsymbol{\Theta}_{i-d-1}  @>f>>  \boldsymbol{\Theta}_{i-1} @>>> \theta_{i-1}@>>>0\\
@. @VV\phi_{i-d}V @VV\phi_iV @VV\psi_iV\\
0 @>>>\boldsymbol{\Theta}_{i-d}^{n+1}  @>\oplus f>>  \boldsymbol{\Theta}_{i}^{n+1} @>>>  \theta_{i}^{n+1}@>>>0
\end{CD},
\end{equation}
\noindent where $\phi_i$ sends every $r \in \boldsymbol{\Theta}_{i-1}$ to $r(x_0,x_1,\cdots, x_n) \in  \boldsymbol{\Theta}^{n+1}_i$ and $\psi_i$ is induced by $\phi_i $ in the natural way. Then the degree $(i-1)$ component of the socle of $\theta$ is just $\ker( \psi_i)$. The map $\phi_i$ is injective since the socle of $ \boldsymbol{\Theta}$ is the one-dimensional vector space over $k$ generated by $x_0^{q-1}\cdots x_n^{q-1}$, an element of degree $> m(q)$. Therefore, we have the injection 
\[\ker{\psi_i} \hookrightarrow (0:\i m)_{\coker (\phi_{i-d})}.\]
While the following lemma guarantees $(0:\i m)_{\coker (\phi_{i-d})}=0$, our theorem follows.
\end{proof}

We fix some notation which are valid for this lemma only. Let $k$ be an arbitrary field in any characteristic. Let $R$ be the standard graded Artinian ring $k[x_0,\cdots, x_n]/(x_0^t, \cdots, x_n^t)$, where $t$ is a fixed positive integer. Let $\i m_R$ denote the maximal ideal $(x_0, \cdots, x_n)$ of $R$.
\begin{lemma} \label{soc}Let $\phi$ be the map from $R$ to $R^{n+1}$ defined by $\phi(r)=r(x_0,x_1,\cdots, x_n)$. Let $M$ be the cokernel of $\phi$. Then the socle degree of $M$ is at least $n(t-1)$, i.e., if $a$ is a nonzero homogeneous element in $(0:\i m_R)_M$, then the degree of $a\geq n(t-1)$. 

\end{lemma}

 \begin{proof} We first point out the following fact, which is easy to verify.
\begin{fact}\label{fact} In $R$, we have $(0:x_i)=(x_i^{t-1}), \forall i$.
\end{fact}

To prove the lemma, we induct on $n$. The case $n=0$ is trivial.
Assume the theorem holds for all Artinian rings of the form $k[x_0,x_1, \cdots, x_l]/({x_0}^t, {x_1}^t, \cdots, {x_l}^t)$ for all $l<n$. Consider the case  $R=k[x_0,x_1, \cdots, x_n]/({x_0}^t, {x_1}^t, \cdots, {x_n}^t)$. We use $R_i$ to denote the degree $i$ component of $R$.
 
Suppose the lemma fails, then there exists an nonzero homogeneous element $\bar a \in  (0:\i m_R)_M$, whose degree is $d < n(t-1)$. Since $M$ is a quotient module of $R^{n+1}$, we can assume $\bar a$ is the image of $a=(a_0, a_1, \cdots, a_n)$ where  $a_0, a_1, \cdots, a_n \in R_d$. The condition $\bar a \in (0:\i m_R)_M$ implies that there exist $u_0, u_1, \cdots, u_n \in R_d$ such that 
\begin{equation}
\label{constrain} ax_i =u_i(x_0, x_1, \cdots, x_n),  0 \leq i \leq n.
\end{equation} 
It follows that for every $j=0,1,\cdots,n$, 
\begin{equation}
\label{constrain2} a_jx_i=u_ix_j, 0 \leq i \leq n.
\end{equation} 
This implies $a_j\i m_R \subseteq (x_j)$, i.e
\begin{equation}\label{ind}
a_j \in (x_j:\i m_R).
\end{equation}
Let $\tilde R$ denote $R/x_jR$, $\tilde a_j$ be the image of $a_j$ in $\tilde R$ and $\i m_{\tilde R}$ the maximal ideal of $\tilde R$. Then (\ref{ind}) becomes
$$\tilde a_j \in (0: \i m_{\tilde R}). $$
Applying Fact \ref{fact} to $\tilde R$, since degree of $a_j < n(t-1)$, we see that $\tilde a_j$ must be the zero in $\tilde R$. This means
\[a_j \in (x_j).\]
Hence, we can assume there exists $a_j' \in R_{d-1}$ such that $a_j=a_j'x_j$, for $j =0,1,\cdots, n$.
Let $a'=a_0'(x_0, \cdots, x_n)$, which is an element in $\im \phi$.  Then $\bar a$ is also the image of $a-a'$. So, we can replace $a$ by $a-a'$ to assume $a$ is of the form $(0, a_1, \cdots, a_n)$  in the first place. Therefore, by (\ref{constrain}), we get
\begin{equation}
u_ix_0=0, 0 \leq i \leq n.
\end{equation}
Multiplying both sides of (\ref{constrain2}) by $x_0$, we then obtain
\begin{equation}
a_jx_ix_0=0,  0 \leq i, j \leq n.
\end{equation} 
Therefore, for  every $j=0,1,\cdots,n$, $a_jx_0 \in (0:\i m_R)$. Thus, by Fact \ref{fact} again, we get
\begin{equation}
a_jx_0=\lambda_jx_0^{t-1}\cdots x_n^{t-1}
\end{equation}
for some $\lambda_j\in k$. This contradicts the fact that $\deg a_j < n(t-1)$.

 \end{proof}

\begin{cor}
\label{formula}

Adopt all of the notation of Theorem~\ref{BC}. Assume $R=S/fS$ attains the minimal Hilbert-Kunz function. Then there exists a constant $c$, such that
\[\lambda (\soc(R/\i m^{[q]})) = cq^{n-2}+O(q^{n-3}).\]
Moreover, the constant $c$ has the following expressions:
\begin{displaymath}
   c = \left\{
     \begin{array}{lr}
       \frac{d((-1)^{n-d}-3)}{2n!}{n \choose 2}  \sum_{i=0}^{\nu-1} (-1)^i {{n+1} \choose i} (\nu-i)^{n-2}, \text{ if } n=2\nu-1 \text{ is odd}\\
       \\
      \frac{d((-1)^{n-d}-3)}{2n!}{n \choose 2}   \sum_{i=0}^{\nu-1} (-1)^i {{n+1} \choose i} (\nu+\frac{1}{2}-i)^{n-2} , \text{ if } n=2\nu \text{ is even}
     \end{array}
   \right.
\end{displaymath} 
\end{cor}

\begin{proof} By Theorem~\ref{main}, it remains to calculate
$\dim_k \theta_{m(q)}$, which is equal to $ \dim_k {\boldsymbol{\Theta}}_{m(q)}-
\dim_k {\boldsymbol{\Theta}}_{m(q-d)}$ by (\ref{ses}). Note that $ \dim_k {\boldsymbol{\Theta}}_i$ is the coefficient of $t^i$ in the polynomial $(1+t+t^2+\cdots+t^{q-1})^{n+1}$. The rest of the calculation is completely elementary, which will be carried out in detail in the next section.
\end{proof}

\begin{remark}
Applying Corollary~\ref{formula} to the 3-dimensional hypersurface $k[x,y,z,w]/(xy-zw)$, we see that  $\lambda (\soc(R/\i m^{[q]}))=4q-3$;  to the Cayley's cubic surface $k[x,y,z,w]/(xyz+xyw+xzw+yzw)$, we get $\lambda (\soc(R/\i m^{[q]}))=3q-3$. One might expect that in general, the limit of $\lambda (\soc(R/\i m^{[q]}))/q^{\max\{0, n-2\}}$ (as $q \to \8$) exists. However, the following example in dimension two provides a negative answer for this question. Let $R$ be the coordinate ring of the rational quintic curve in $\mathbb P^3$ parametrized by $(s,t) \to (t^5, st^4, s^4t,s^5)$ in characteristic 2. Our Macaulay 2 experiment gives the following eventually periodic sequence for the lengths of $\soc(R/\i m^{[q]})$:
\begin{center}
\begin{tabular}{|c|c|c|c|c|c|c|c|c|c|c|c|c|c|}
\hline
$q$&1&2&3&4&5&6&7&8&9&10&11&12&...\\
\hline
socle length&5&9&13&21&19&17&17&21&19&17&17&21&...\\
\hline
\end{tabular}
\end{center}
The author thanks Jason McCullough for providing some help on Macaulay 2 programming.
\end{remark}

\begin{remark} This remark is due to Florian Enescu and Yongwei Yao. Let $(R, \i m)$ be a local ring and $S=R[x]_{(\i m, x)}$ with maximal ideal $\i n=(\i m, x)$. Then
$\length (\soc(R/\i m^{[q]}))= \length (\soc(S/\i n^{[q]}))$
for all $q$. We leave the verification of this to the interested readers.
\end{remark}

\section{Two combinatorial identities}\label{cl}
This entire section is elementary. We prove a combinatorial result Theorem~\ref{elementary}, which contains the calculation we mentioned in the proof of Corollary~\ref{formula} as a special case. 

Fix an integer $n>0$, consider the following function
\begin{equation}\label{genfun}
\sum_{i\geq 0} \GG (i)t^i =(1+t+t^2+\cdots+t^{q-1})^{n+1}.
\end{equation} 
Assume  $m(q)$ is an integer-valued function of the following form
\[m(q)=\bigg(\dfrac{n+1}{2}\bigg )q+\xi + \epsilon\]
where $\xi$ is a constant and $\epsilon$ is defined according to
\begin{displaymath}
   \epsilon = \left\{
     \begin{array}{lr}
       0,& \text{ if  $n$ is odd,}\\
    (1/2)(q-2\lfloor{q/2}\rfloor),& \text{ if  $n$ is even.}
     \end{array}
   \right.
\end{displaymath} 
For any fixed integer $d>0$ define
\begin{equation}
h(q)=\GG \bigg(m(q) \bigg)-\GG \bigg(m(q)-d \bigg).
\end{equation}
Then,  we have the following estimate about $h(q)$:
\begin{theorem}\label{elementary}
There exists a constant $c$, such that
 \[h(q)=cq^{n-2}+O(q^{n-3}).\]
Moreover, the constant $c$ has the following expressions
\begin{displaymath}
   c = \left\{
     \begin{array}{lr}
       \dfrac{d}{n!}{n \choose 2} (2\xi+2\epsilon+n-d+1) \sum_{i=0}^{\nu-1} (-1)^i {{n+1} \choose i} (\nu-i)^{n-2},& \text{ if } n=2\nu-1\\
      \dfrac{d}{n!}{n \choose 2} (2\xi+2\epsilon+n-d+1) \sum_{i=0}^{\nu-1} (-1)^i {{n+1} \choose i} \bigg (\nu+\dfrac{1}{2}-i\bigg)^{n-2} ,& \text{ if } n=2\nu
     \end{array}
   \right.
\end{displaymath} 
\end{theorem} 
\begin{proof}
By binomial theorem, we have
\begin{align}
&(1+t+t^2+\cdots+t^{q-1})^{n+1}
=\dfrac{(1-t^q)^{n+1}}{(1-t)^{n+1}} \notag\\
=&\bigg(\sum_{i=0}^{n+1}(-1)^i{{n+1} \choose i}t^{qi}\bigg) \bigg(\sum_{i=0}^{\infty}{{n+i} \choose n}t^i \bigg )\label{Hilbert} 
\end{align}
\noindent {\bf Case 1. }   $n=2\nu-1$. 
Then (we leave $\epsilon$ here for the purpose of Case 2, even though it is 0 here)
\[m(q)=\nu q +\xi+\epsilon.\]
By comparing the coefficients of $t^{m(q)}$ in (\ref{genfun}) and (\ref{Hilbert}), we get
\begin{equation}
\GG \bigg(m(q) \bigg)=\sum_{i=0}^{\nu-1}(-1)^i
{{n+1} \choose i}
{{(\nu-i)q+\xi+\epsilon+n} \choose n}
\end{equation}
and
\begin{equation}
\GG \bigg(m(q)-d \bigg)=\sum_{i=0}^{\nu-1}(-1)^i
{{n+1} \choose i}
{{(\nu-i)q+\xi+\epsilon+n-d} \choose n}.
\end{equation}
Let
\begin{equation}\label{GQ}
g(q)={{(\nu-i)q+\xi+\epsilon+n} \choose n}-{{(\nu-i)q+\xi+\epsilon+n-d} \choose n},
\end{equation}
then, one could rewrite $h(q)$ as
\begin{equation}
h(q)=\sum_{i=0}^{\nu-1}(-1)^i
{{n+1} \choose i}
g(g).
\end{equation}
To estimate $g(q)$, we use \textit{Stirling numbers of the first kind} $s(n,k)$ to expand $g(q)$ as a polynomial of $(\nu-i)q$. Recall  by definition,  $s(n,k)$ is the coefficient of 
$x^k$ in the polynomial $x(x-1)\cdots (x-n+1)$, i.e,
\begin{equation}\label{Stirling}
x(x-1)\cdots (x-n+1)=\sum_{k=0}^n s(n,k)x^k.
\end{equation}
Therefore, for any integer $Z$, we have
\begin{align*}
&n!{(\nu-i)q+Z \choose n}\\
=&\sum_{k=0}^n s(n,k)\bigg ((\nu-i)q+Z\bigg)^k\\
=&\sum_{k=0}^n s(n,k)\bigg (\sum_{j=0}^k{k \choose j}((\nu-i)^jq^jZ^{k-j})\bigg )\\
=&(\nu-i)^nq^n+\bigg( {n \choose 1}((v-i)q)^{n-1}Z-{n \choose 2}((v-i)q)^{n-1} \bigg)\\
\hspace{1cm}&+\bigg({n \choose 2} ((v-i)q)^{n-2}Z^2 -{n \choose 2}{n-1 \choose 1}((v-i)q)^{n-2}Z +s(n,n-2)((v-i)q)^{n-2} \bigg)+o(q^{n-3}).
\end{align*}
Here in the last equality, we use the fact $s(n,n)=1$ and $s(n,n-1)=-{n \choose 2}$.
It follows that
\begin{align*}
n!g(q)&=d((v-i)q)^{n-1}{n \choose 1}+d((v-i)q)^{n-2}\bigg({n \choose 2}  (2\xi+2\epsilon+2n-d)-{n \choose 2}{n-1 \choose 1} \bigg)+o(q^{n-3})\\
&=d((v-i)q)^{n-1}{n \choose 1}+d((v-i)q)^{n-2}{n \choose 2}  (2\xi+2\epsilon+n-d+1)+o(q^{n-3}).
\end{align*}
 Therefore
 \begin{align*}
 h(q)=&q^{n-1}\bigg( \dfrac{d}{n!}{n \choose 1}\sum_{i=0}^{\nu-1} (-1)^i {{n+1} \choose i} (\nu-i)^{n-1} \bigg)+\\
 &q^{n-2}\bigg(\dfrac{d}{n!}{n \choose 2} (2\xi+2\epsilon+n-d+1) \sum_{i=0}^{\nu-1} (-1)^i {{n+1} \choose i} (\nu-i)^{n-2}       \bigg)+o(q^{n-3}).
 \end{align*} 
 Hence, by (\ref{C1}), the coefficient of $q^{n-1}$ in $h(q)$ is 0. Moreover, 
 \[c=\dfrac{d}{n!}{n \choose 2} (2\xi+2\epsilon+n-d+1) \sum_{i=0}^{\nu-1} (-1)^i {{n+1} \choose i} (\nu-i)^{n-2}. \]
 
\noindent {\bf Case 2. }   $n=2\nu$. 
In this case,
\[m(q)=\bigg (\nu+\frac{1}{2} \bigg)q+\xi+\epsilon.\]
Exactly the same computations yield
\begin{align*}
 h(q)=&q^{n-1}\bigg( \dfrac{d}{n!}{n \choose 1}\sum_{i=0}^{\nu-1} (-1)^i {{n+1} \choose i} (\nu+\dfrac{1}{2}-i)^{n-1} \bigg)+\\
 &q^{n-2}\bigg(\dfrac{d}{n!}{n \choose 2} (2\xi+2\epsilon+n-d+1) \sum_{i=0}^{\nu-1} (-1)^i {{n+1} \choose i} (\nu+\dfrac{1}{2}-i)^{n-2}       \bigg)+o(q^{n-3}).
 \end{align*} 
 So  by (\ref{C2}), the coefficient of $q^{n-1}$ in $h(q)$ is 0 and
  \[c=\dfrac{d}{n!}{n \choose 2} (2\xi+2\epsilon+n-d+1) \sum_{i=0}^{\nu-1} (-1)^i {{n+1} \choose i} \bigg (\nu+\dfrac{1}{2}-i\bigg)^{n-2}. \]
\end{proof}

\begin{lemma} For any positive integer $n$, the following identities hold:
\begin{equation}\label{C1}
\sum_{i=0}^n (-1)^i {{2n} \choose i} (n-i)^{2n-2}=0
\end{equation}
and
\begin{equation}\label{C2}
\sum_{i=0}^n (-1)^i {{2n+1} \choose i} \bigg (n-i+\frac{1}{2} \bigg)^{2n-1}=0.
\end{equation}
\end{lemma}

\begin{proof}
We only prove (\ref{C1}). The proof of (\ref{C2}) is similar. The following elementary proof of (\ref{C1}) is suggested by Daniel Smith-Tone to the author. First we notice that 
\begin{align*}
2\sum_{i=0}^n (-1)^i {{2n} \choose i} (n-i)^{2n-2}&=\sum_{i=0}^n (-1)^i {{2n} \choose i} (n-i)^{2n-2}+\sum_{j=n+1}^{2n} (-1)^j {{2n} \choose j} (n-j)^{2n-2}\\
&=\sum_{i=0}^{2n} (-1)^i {{2n} \choose i} (n-i)^{2n-2}.
\end{align*}

For any function $f(x)$,  one defines $\bigtriangleup f(x)$, the \textit{forward difference} of $f(x)$, to be the function $f(x+1)-f(x)$. The higher order forward difference is defined recursively by $\bigtriangleup^nf(x)=\bigtriangleup^{n-1}(\bigtriangleup f(x))$. It is then easy to check that for any positive integer $k$, 
$$\bigtriangleup^k f(x)=\sum_{i=0}^k (-1)^i{ k \choose i}f(x+k-i). $$
Apply the above to $f(x)=x^{2n-2}$. Since $f(x)$ is a polynomial of degree $2n-2$, $\bigtriangleup^{2n}f(x)$ must be zero, i.e.
$$\sum_{i=0}^{2n} (-1)^i{ 2n \choose i}(x+2n-i)^{2n-2}=0.$$
In particular, we can take $x=-n$ and the identity is proved.
\end{proof}

\section{Socle length and Betti number}\label{betti}
 In this section, we provide some connections between the socle length function we considered in Section~\ref{sl} and the asymptotic growth of some other invariants in characteristic $p$, such as Betti numbers. The main result of this section is 
 \begin{theorem}\label{same}
 Let $(R, \i m, k)$ be a local ring in characteristic $p$. Let $I$ and $\i a$ be  $\i m$-primary ideals. Suppose $I=J+uR$ for some $u \in R$, and $J$ is an ideal of finite projective dimension which satisfies the condition that $R/J^{[q]}$ is Artinian Gorenstein for $q \gg0$. Then the differences between any of the following two numerical functions (as functions on $q$) are bounded as $q \to \8$:
\begin{itemize}
\item[(1)] $\lambda (\hom(R/\i a, R/I^{[q]}))$;
\item[(2)] $\lambda (\tor_1(R/I, {}^{f^n}\!\! R)\otimes R/ \i a)$;
\item[(3)] $\lambda (\tor_2(R/I^{[q]}),R/ \i a)$.
\end{itemize} 
\end{theorem}

\begin{proof}
We first show for a given $R$-module $M$, 
\begin{equation}\label{fact1}
\length (\hom (R/ \i a, M)) = \length( \hom(M,E) \otimes R/ \i a),
\end{equation}
where $E=E(k)$ is the injective hull of $k$.
To see this, suppose $\i a=(a_1, \cdots, a_c)$ and consider the following exact sequence
\begin{equation}\label{amap}
0 \to \hom (R/ \i a, M) \to M \to \bigoplus_1^c M,
\end{equation}
where the rightmost map sends every $\tau \in M$ to $(a_1\tau, \cdots, a_c \tau) \in \bigoplus_1^c M$.
Then the equality (\ref{fact1}) follows easily from taking the Matlis dual of (\ref{amap}).

Applying(\ref{fact1}) to the case $M=R/I^{[q]}$ (take $q$ large enough such that $ J^{[q] }\subseteq \i a$),
since $R/J^{[q]}$ is Gorenstein, $E=R/J^{[q]}$, we obtain
\begin{equation*}
\length(\hom(R/\i a,R/I^{[q]}))=\length(\hom (R/I^{[q]}, R/J^{[q]}) \otimes R/\i a).
\end{equation*}
Observe that
\begin{equation*}
\hom (R/I^{[q]}, R/J^{[q]}) \cong \dfrac{J^{[q]}:I^{[q]}}{J^{[q]}}=\dfrac{J^{[q]}:u^q}{J^{[q]}}.
\end{equation*}
The last equality here is due to our assumption $I=J+uR$. We therefore have
\begin{equation}\label{twolength}
\length(\hom(R/\i a,R/I^{[q]}))=\length \bigg(\dfrac{J^{[q]}:u^q}{J^{[q]}}  \otimes R/\i a \bigg).
\end{equation}
Again, since $I=J+uR$, we have a short exact sequence
\begin{equation}\label{basicses}
0 \to \dfrac{R}{J:u} \to \dfrac{R}{J} \to \dfrac{R}{I} \to 0.
\end{equation}
Tensoring (\ref{basicses}) with $ {}^{f^n}\!\! R$. Notice that $\tor_1(R/J,  {}^{f^n}\!\! R)=0$ since $J$ has finite projective dimension,
we see there exists an exact sequence
\begin{equation}\label{torseq}
0 \to \tor_1(R/I, {}^{f^n}\!\! R) \to \dfrac{R}{(J:u)^{[q]}} \to \dfrac{R}{J^{[q]}} \to \dfrac{R}{I^{[q]}}\to 0.
\end{equation}
Comparing the exact sequence (\ref{torseq}) with the short exact sequence
\[0 \to \dfrac{R}{J^{[q]}:u^q} \to \dfrac{R}{J^{[q]}} \to \dfrac{R}{I^{[q]}} \to 0,\]
which is obtained in a way similar to (\ref{basicses}),
we have
\[\tor_1(R/I, {}^{f^n}\!\! R) \cong \dfrac{J^{[q]}:u^q}{(J:u)^{[q]}}. \]
Hence we obtain a short exact sequence
\begin{equation}\label{firstiso}
0 \to \dfrac{(J:u)^{[q]}}{J^{[q]}} \to \dfrac{J^{[q]}:u^q}{J^{[q]}} \to\tor_1(R/I, {}^{f^n}\!\! R)  \to 0.
\end{equation}
Tensoring (\ref{firstiso}) with $R/ \i a$ gives rise to an exact sequence
\[\to \dfrac{(J:u)^{[q]}}{J^{[q]}} \otimes R/ \i a \to \dfrac{J^{[q]}:u^q}{J^{[q]}} \otimes R/ \i a \to \tor_1(R/I, {}^{f^n}\!\! R)\otimes R/ \i a  \to 0.\]
It follows that 
\[0 \leq \length \bigg(\dfrac{J^{[q]}:u^q}{J^{[q]}}  \otimes R/\i a \bigg)-\length (\tor_1(R/I, {}^{f^n}\!\! R)\otimes  R/ \i a) \leq \length \bigg(\dfrac{(J:u)^{[q]}}{J^{[q]}} \otimes R/ \i a\bigg). \]
Thus from (\ref{twolength}), we conclude that
\[0 \leq \length (\hom(R/ \i a, F^n(R/I)))-\length (\tor_1(R/I, {}^{f^n}\!\! R)\otimes  R/ \i a) \leq \length \bigg(\dfrac{(J:u)^{[q]}}{J^{[q]}} \otimes R/ \i a\bigg).\]
But the right hand side is 
\[ \leq \length \bigg(\dfrac{(J:u)^{[q]}}{J^{[q]}} \otimes k\bigg) \length(R/ \i a)\leq \length \bigg(\dfrac{(J:u)}{J} \otimes k\bigg) \length(R/ \i a) =O(1).\]
Therefore, the difference between (1) and (2) is bounded as $q \to \8$.

To establish the boundedness of the difference between (2) and (3), we use some spectral sequence arguments.
Let $F_\bullet$ be the minimal free resolution of $R/I$ and $G_\bullet$ the minimal free resolution of $R/\i a$.  The double complex $F^n(F_\bullet) \otimes G_\bullet$ yields the following spectral sequence
\[\tor_i (\tor_j(R/I, {}^{f^n}\!\! R), R/\i a) \Rightarrow H_{i+j}(F^n(F_\bullet) \otimes R/ \i a).\]
From the exact sequence of low degree terms, we get the following exact sequence
\[H_2(F^n(F_\bullet) \otimes R/ \i a) \to \tor_2(F^n(R/I),R/ \i a) \to \tor_1(R/I, {}^{f^n}\!\! R)\otimes R/ \i a  \overset{0}{\to} H_1(F^n(F_\bullet) \otimes R/ \i a)  \to\]
\[ \overset{\varsigma}{\to}  \tor_1(F^n(R/I),R/ \i a) \to 0.\]
In this exact sequence, we choose $n \gg 0$ so that $\i m^{[q]} \subseteq \i a$, which forces the map $\varsigma$ to be an isomorphism. It then follows that
\begin{equation}\label{lowdegree}
0 \leq \length(\tor_2(F^n(R/I),R/ \i a)) - \length (\tor_1(R/I, {}^{f^n}\!\! R)\otimes R/ \i a) \leq \length(H_2(F^n(F_\bullet) \otimes R/ \i a)).
\end{equation}
Since $\i m^{[q]} \subseteq \i a$, we also have
\[H_{i}(F^n(F_\bullet) \otimes R/ \i a) =\bigoplus_1^{\rank F_i} R/ \i a .\]
So the right hand side of (\ref{lowdegree}) equals $(\rank F_2)\length(R/\i a)$, which is independent of $q$.
\end{proof}
 
 We point out here that in Theorem~\ref{same}, if we take $\i a$ to be the maximal ideal $\i m$, then (1) gives us the socle length function we considered in Section~\ref{sl} and (3) gives the second Betti numbers of $R/I^{[q]}$. 

\specialsection*{ACKNOWLEDGEMENTS}
We thank the referee's many suggestions which make the presentation of this manuscript a lot better than before. 



\end{document}